\documentclass{amsart}
\usepackage{amsfonts,amssymb,amsmath,amsthm, latexsym}
\usepackage{url}
\usepackage{enumerate}
\usepackage{graphicx}
\usepackage[dvipsnames]{xcolor}

\newtheorem{theorem}{Theorem}[section]
\newtheorem{lemma}[theorem]{Lemma}
\newtheorem{proposition}[theorem]{Proposition}

\theoremstyle{definition}
\newtheorem{definition}[theorem]{Definition}

\numberwithin{equation}{section}

\author{J. Janssen}
\address{Department of Mathematics and Statistics, Dalhousie University, Canada}
\email{Jeannette.Janssen@dal.ca}
\author{A. Quas}
\address{Department of Mathematics and Statistics, University of Victoria, Canada}
\email{aquas(a)uvic.ca}
\author{R. Yassawi}
\address{Department of Mathematics, Trent University,  Canada}
\email{ryassawi@trentu.ca}
\thanks{ The first two authors are partially supported by NSERC Discovery Grants. }


\usepackage{amsthm,amssymb,amsmath,latexsym}

\newtheorem{example}[theorem]{Example}
\DeclareMathOperator{\Var}{Var}
\newcommand{\E}{\mathbb E}
\newcommand{\PP}{\mathbb P}

\newcommand{\N}{\mathbb N}
\newcommand{\om}{\omega}
\newcommand{\wt}{\widetilde}

\begin{document}
\title{Bratteli diagrams where random orders are 
imperfect}

\begin{abstract}
For the simple Bratteli diagrams $B$ where there is a single edge connecting 
any two vertices in consecutive levels, we show that a random order has 
uncountably many infinite paths if and only if the growth rate of the 
level-$n$ vertex sets is super-linear.  This gives us the dichotomy: a 
random order on a slowly growing Bratteli diagram admits a homeomorphism, 
while a random order on a quickly growing Bratteli diagram does not.
We also show that for a large family of infinite rank Bratteli diagrams $B$,
a random order on $B$ does not admit a continuous Vershik map.
\end{abstract}

\maketitle

\section{Introduction}\label{Introduction}
Consider the following random process. For each natural number  $n$, we have a 
collection of finitely many individuals.
Each individual in the $n+1$-st collection randomly picks a parent from the 
$n$-th collection, and this is done for all $n$.  If we know how many 
individuals there are in each generation, the question ``How many infinite 
ancestral lines are there?" almost always has a common answer $j$: 
what is it? We can also make this game more general, by for each 
individual, changing the odds that he choose a certain parent, 
and ask the same question.

The information that we are given will come as a {\em Bratteli diagram} 
$B$ (Definition \ref{Definition_Bratteli_Diagram}), where each 
``individual" in generation $n$ is represented by a vertex in the $n$-th 
vertex set $V_n$, and the chances that an individual $v\in V_{n+1}$ 
chooses $v' \in V_n$ as a parent is the ratio of the number of edges 
incoming to $v$ with source $v'$ to the total number of edges incoming 
to $v$. We  consider the space $\mathcal O_B$ of {\em orders} on $B$ 
(Definition \ref{order_definition}) as a measure space equipped with 
the completion of the uniform product measure $\mathbb P$.
A result in \cite{bky} (stated as Theorem \ref{generic_theorem} here) 
tells us that there is some $j$, either a positive integer or infinite, 
such that a $\mathbb P$-random order $\om$ possesses $j$ maximal paths.

Bratteli diagrams, which were first studied in operator algebras, appeared 
implicitly  in the measurable dynamical setting in \cite{vershik_1,vershik}, 
where it was shown that any 
ergodic invertible transformation of a Lebesgue space can be represented  
as a  measurable ``successor" (or {\em Vershik}) map on the space of 
infinite paths $X_B$ in some Bratteli diagram $B$ (Definition 
\ref{Good_and_bad}). The successor map, which  is defined using 
an order on $B$, is not defined on the set of maximal paths in 
$X_B$, but as this set  is typically a null set, it poses no 
problem in the measurable framework.  Similar results were 
discovered in the topological setting in \cite{hps}: any minimal 
homeomorphism  on a Cantor Space has a representation as a  (continuous, 
invertible) Vershik map which is  defined on all of $X_B$ for some 
Bratteli diagram $B$. To achieve this, the technique used in 
\cite{hps} was to construct the order so that it had a unique 
minimal and maximal path, in which case the successor map extends 
uniquely to a homeomorphism of $X_B$.  For such an order our 
quantity $j$ takes the value 1.  We were curious to see whether such an order is 
typical, and whether a typical order defined a continuous Vershik 
map. 
Note that the value $j$ 
is not an invariant of topological dynamical properties that 
are determined by the Bratteli diagram's {\em dimension group} 
\cite{effros}, such as {\em strong orbit equivalence} 
\cite{g_p_s}.

In this article we compute  $j$ for a large family of {\em infinite rank}
Bratteli diagrams (Definition \ref{rank_d_definition}). Namely, in Theorem 
\ref{thm:dichot}, we show that $j$ is uncountable for the situation where 
any individual at stage $n$ is equally likely to be chosen as a parent by 
any individual at stage $n+1$, whenever  the generation growth rate is super-linear. 
If the generations grow at a slower rate than this, $j=1$. 
We note that this latter situation has been studied in the context of 
gene survival in a variable size population, as in the Fisher-Wright model 
(e.g. \cite{seneta}, \cite{donnelly}). We describe this connection in 
Section \ref{results}.

In Theorem \ref{general_case} we generalise part of Theorem  \ref{thm:dichot}
to a large family of  Bratteli diagrams.
We can draw the following conclusion from these results.  
An order $\om$ is called {\em perfect} if it admits a continuous 
Vershik map.  
Researchers working with continuous Bratteli-Vershik 
representations of dynamical systems usually work with a subfamily of perfect orders called
{\em proper} orders: those that have only one maximal and one minimal path.
For a large class of simple Bratteli 
diagrams (including the ones we identify in Theorems \ref{special_case} and 
\ref{general_case}), if $j>1$, then  a $\mathbb P$-random order is almost 
surely not perfect (Theorem \ref{random_order_imperfect}), hence not proper. 
We note that this is in 
contrast to the case for finite rank diagrams, where almost any order put on   
any reasonable finite rank Bratteli diagram is 
perfect (Section 5, \cite{bky}). 

\section{Bratteli diagrams and Vershik maps}\label{Preliminaries}

In this section, we collect the notation and basic definitions that
are used throughout the paper.

\subsection{Bratteli diagrams}
\begin{definition}\label{Definition_Bratteli_Diagram}
A {\it Bratteli diagram} is an infinite graph $B=(V,E)$ such that the vertex
set $V=\bigcup_{i\geq 0}V_i$ and the edge set $E=\bigcup_{i\geq 1}E_i$
are partitioned into disjoint subsets $V_i$ and $E_i$ where

\begin{enumerate}[(i)]
\item $V_0=\{v_0\}$ is a single point;
\item $V_i$ and $E_i$ are finite sets;
\item there exists a range map $r$ and a source map $s$, both from $E$ to
$V$, such that $r(E_i)= V_i$, $s(E_i)= V_{i-1}$.
\end{enumerate}

\end{definition}

Note that $E$ may contain multiple edges between a pair of vertices.
The pair $(V_i,E_i)$ or just $V_i$ is called the {\em $i$-th level} of the
diagram $B$. A finite or infinite sequence of edges $(e_i : e_i\in E_i)$
such that $r(e_{i})=s(e_{i+1})$ is called a {\it finite} or {\it infinite
path}, respectively. 

For $m<n$, $v\, \in V_{m}$ and $w\,\in V_{n}$, let $E(v,w)$ denote the set of 
all paths $\overline{e} = (e_{1},\ldots, e_{p})$
with $s(e_{1})=v$ and $r(e_{p})=w$. 
For a Bratteli diagram $B$,
let $X_B$ be the set of infinite paths starting at the top vertex $v_0$.
 We
endow $X_B$ with the topology generated by cylinder sets
$\{U(e_j,\ldots,e_n): j, \,\, n \in \mathbb N, \mbox{ and }
 (e_j,\ldots,e_n) \in E(v, w),  v \in V_{j-1}, w \in V_n
\}$, where
$U(e_j,\ldots,e_n):=\{x\in X_B : x_i=e_i,\;i=j,\ldots,n\}$.
With this topology, $X_B$ is a 0-dimensional compact metrizable space.

\begin{definition}\label{incidence_matrices_definition}
Given a Bratteli diagram $B$, the $n$-th {\em incidence matrix}
$F_{n}=(f^{(n)}_{v,w}),\ n\geq 0,$ is a $|V_{n+1}|\times |V_n|$
matrix whose entries $f^{(n)}_{v,w}$ are equal to the number of
edges between the vertices $v\in V_{n+1}$ and $w\in V_{n}$, i.e.
$$
 f^{(n)}_{v,w} = |\{e\in E_{n+1} : r(e) = v, s(e) = w\}|.
$$
\end{definition}

\begin{definition}\label{rank_d_definition} Let $B$ be a Bratteli diagram.
\begin{enumerate}
\item
We say $B$ has \textit{finite rank} if  for some $k$, $|V_n| \leq k$ for all $n\geq 1$.
\item
We say that $B$ is {\em simple} if  for any level
$m$ there is $n>m$ such that $E(v,w) \neq \emptyset$ for all $v\in
V_m$ and $w\in V_n$.
\item
We say that 
a Bratteli diagram is {\em completely connected} if all entries of its incidence 
matrices are positive.

\end{enumerate}
\end{definition}

In this article we work only with completely connected Bratteli diagrams.

\subsection{Orderings on a Bratteli diagram}

\begin{definition}\label{order_definition} A Bratteli diagram $B=(V,E) $ 
is called {\it ordered}
if a linear order ``$>$" is defined on every set  $r^{-1}(v)$, $v\in
\bigcup_{n\ge 1} V_n$. We use  $\om$ to denote the corresponding partial
order on $E$ and write $(B,\om)$ when we consider $B$ with the ordering 
$\om$. Denote by $\mathcal O_{B}$ the set of all orderings on $B$.
\end{definition}

Every $\omega \in \mathcal O_{B}$ defines a \textit{lexicographic}
partial ordering on the set of finite paths between
vertices of levels $V_k$ and $V_l$:  $(e_{k+1},\ldots,e_l) > (f_{k+1},\ldots,f_l)$
if and only if there is an $i$ with $k+1\le i\le l$, $e_j=f_j$ for $i<j\le l$
and $e_i> f_i$.
It follows that, given $\om \in \mathcal O_{B}$, any two paths from $E(v_0, v)$
are comparable with respect to the lexicographic ordering generated by $\om$.  
If two infinite paths  are {\em tail equivalent}, i.e. agree from some vertex 
$v$ onwards, then we can compare them by comparing their initial segments in
$E(v_0,v)$. Thus $\om$ defines a partial order on $X_B$, where two infinite 
paths are comparable if and only if they are tail equivalent.

\begin{definition}
We call a finite or infinite path $e=(e_i)$ \textit{ maximal (minimal)} if every
$e_i$ is maximal (minimal) amongst the edges
from $r^{-1}(r(e_i))$.
\end{definition}

Notice that, for $v\in V_i,\ i\ge 1$, the
minimal and maximal (finite) paths in $E(v_0,v)$ are unique. Denote
by $X_{\max}(\om)$ and $X_{\min}(\om)$ the sets of all maximal and
minimal infinite paths in $X_B$, respectively. It is not hard to show that
$X_{\max}(\om)$ and $X_{\min}(\om)$ are non-empty closed subsets of
$X_B$. If $B$ is completely connected, then  $X_{\max}(\om)$ and $X_{\min}(\om)$  have no
interior points.

Given a Bratteli diagram $B$, we can  describe the set
of all orderings $\mathcal O_{B}$ in the following way.
Given a  vertex $v\in V\backslash V_0$, let   $P_v$ denote the set of all 
orders on $r^{-1}(v)$; an element in $P_v$
is denoted by $\om_v$.
Then $\mathcal O_{B}$ can be
represented as
\begin{equation}\label{orderings_set}
\mathcal O_{B} = \prod_{v\in V\backslash V_0}P_v .
\end{equation}
We write an element of $\mathcal O_B$ as $(\omega_v)_{v\in V\setminus V_0}$.

Recall that an $N$th level \emph{cylinder set} is a set of the form
$\bigcap_{v\in \bigcup_{i=1}^{N} V_i}[w_{v}^*]$,
where  $[w_v^*]=\{ \om: \om_v=\om_v^{*} \}$. The collection of $N$th
level cylinder sets forms a finite $\sigma$-algebra, $\mathcal F_N$. 
We let $\mathcal B$ denote the $\sigma$-algebra generated by  $\bigcup_N\mathcal F_N$
and equip $(\mathcal O_B,\mathcal B)$ with the product measure, $\mathbb P'=
\prod_{v\in V\backslash V_0} \mathbb P_v$ where  $\mathbb P_v$
is the uniform measure on $P_v$: $\mathbb P_v(\{i\}) =  
(|r^{-1}(v)|!)^{-1}$ for every $i \in P_v$ and $v\in V\backslash V_0$.
Finally, it will be convenient to extend the measure space $(\mathcal O_B,
\mathcal B,\mathbb P')$ to its completion, $(\mathcal O_B,\mathcal F,\mathbb P)$.
(The reason for the use of the completion is that the subset of $\mathcal O_B$
consisting of orders with uncountably many maximal paths may not be 
$\mathcal B$-measurable, but will shown to be $\mathcal F$-measurable.)

\begin{definition} \label{telescoping_definition}
Let $B$ be a Bratteli diagram, and $n_0 = 0 <n_1<n_2 < \ldots$ be a strictly 
increasing sequence of integers. The {\em telescoping of $B$ to $(n_k)$} 
is the Bratteli diagram $B'$, whose $k$-level vertex set $V_k'= V_{n_k}$ 
and whose incidence matrices $(F_k')$ are defined by
\[F_k'= F_{n_{k+1}-1} \circ \ldots \circ F_{n_k},\]
where $(F_n)$ are the incidence matrices for $B$.
\end{definition}

If $B'$ is a telescoping of $B$, then there is a natural injection 
$L: \mathcal O_B \rightarrow \mathcal O_{B'}$.
Note that unless $|V_n|=1$ for all but finitely many $n$,  
$L(\mathcal O_{B})$ is a set of zero measure in  
$\mathcal O_{B'}$.

\subsection{Vershik maps}

\begin{definition}\label{measvmap}
Let $(B, \omega)$ be an  ordered Bratteli
diagram. The \emph{successor map}, $s_\omega$ is the map
from $X_B\setminus X_\text{max}(\omega)$ to 
$X_B\setminus X_\text{min}(\omega)$ defined by
$s_\omega(x_1,x_2,\ldots)=(x_1^0,\ldots,x_{k-1}^0,\overline
{x_k},x_{k+1},x_{k+2},\ldots)$, where $k=\min\{n\geq 1 : x_n\mbox{
is not maximal}\}$, $\overline{x_k}$ is the successor of $x_k$ in
$r^{-1}(r(x_k))$, and $(x_1^0,\ldots,x_{k-1}^0)$ is the minimal path
in $E(v_0,s(\overline{x_k}))$.
\end{definition}

\begin{definition}\label{VershikMap}
Let $(B, \omega)$ be an  ordered Bratteli
diagram. We say that $\varphi = \varphi_\omega : X_B\rightarrow X_B$
is a {\it continuous Vershik map} if it satisfies the following conditions:
\begin{enumerate}[i]
\item $\varphi$ is a homeomorphism of the Cantor set $X_B$;
\item $\varphi(X_{\max}(\omega))=X_{\min}(\omega)$;
\item $\varphi(x)=s_\omega(x)$ for all $x\in X_B\setminus X_\text{max}(\omega)$.
\end{enumerate}
\end{definition}

If there is an $s_\omega$-invariant measure $\mu$ on $X_B$ such that 
$\mu(X_{\max}(\om))=\mu(X_{\min}(\om))=0$, then we may extend $s_\om$ to a measure-preserving 
transformation $\phi_\om$ of $X_B$. In this case, we call $\phi_\om$ 
a \emph{measurable Vershik map} of $(X_B,\mu)$.
Note that in our case, $X_{\max}(\om)$ and $X_{\min}(\om)$ have empty interiors, 
so that there is at most one continuous extension of the successor map
to the whole space.

\begin{definition}\label{Good_and_bad}  Let $B$ be a Bratteli diagram.
We say that an ordering $\om\in \mathcal O_{B}$ is \textit{perfect}
if $\om$ admits a continuous Vershik map $\varphi_{\om}$ on $X_B$.
If $\om$ is not perfect, we call it  \textit{imperfect}.
\end{definition}

Let $\mathcal P_B\subset \mathcal O_B$ denote the set of perfect orders on $B$.

\section{The size of certain sets in $\mathcal O_B$.}\label{genericity_results}

A finite rank version of the following result was shown in \cite{bky}
as a corollary of the Kolmogorov 0--1 law; 
the proof for non-finite rank diagrams is the same.

\begin{theorem}\label{generic_theorem}
Let $B$ be a simple Bratteli diagram. Then there
exists  $j \in \mathbb N \cup \{\infty\}$ such that for 
$\mathbb P$-almost all orderings,
$|X_{\max}(\om)|=j$.
\end{theorem}

By symmetry (since the sets $X_{\text{max}}(\om)$ and $X_{\text{min}}(\om)$
have the same distribution), it follows that $|X_{\text{min}}(\om)|=j$
almost surely as well.

\begin{example} It is not difficult, though contrived,  to find a  simple 
finite rank Bratteli diagram $B$ where almost all orderings are not perfect. 
Let $V_n=V= \{v_{1},v_{2}\}$ for $n\geq 1$,  and define
$m^{(n)}_{v,w}:=\frac{f^{(n)}_{v,w}}{\sum_{w}f^{(n)}_{v,w}}$: i.e.
$m^{(n)}_{v,w}$ is the proportion of edges with range $v\in
V_{n+1}$ that have source $w\in V_{n}$.
Suppose that $\sum_{n=1}^{\infty} m_{v_{i},v_{j}}^{(n)}<\infty$ for $i\neq j$. 
Then for almost all orderings, there is some $K$ such that for $n>K$, the 
sources of the two maximal/minimal edges at level $n$ are distinct, i.e. $j=2$. 
The assertion follows from   \cite[Theorem~5.4]{bky}.
\end{example}

We point out that given an unordered Bratteli diagram, $B$, if $B$ is 
equipped with two proper  orderings $\omega$ and $\omega'$, then the resulting 
topological Vershik dynamical systems $s$ and $s'$ are 
strong orbit equivalent \cite{g_p_s}.
Likewise, if $B$ equipped with a perfect ordering is telescoped, the 
topological Vershik systems are conjugate. 
The number of maximal 
paths that a random order on $B$  possesses is not invariant under telescoping.
Take for example a Bratteli diagram $B$ where odd levels consist of a unique 
vertex and even levels have $n^2$
vertices. let all incidence matrices have all entries equal to 1. 
By Theorem   \ref{thm:dichot}, a random 
order on $B$ has infinitely many maximal paths. On the other hand, $B$ can 
be telescoped to a diagram $B'$ with only one vertex at each level,  
so that $B'$ has a unique maximal path.

A finite rank version of the following result is proved in Theorem 5.4 of 
\cite{bky}. 

\begin{theorem}\label{random_order_imperfect}
Suppose that $B$ is a  completely connected Bratteli diagram of infinite 
rank so that $\mathbb P$-almost all orderings
have $j$ maximal and minimal elements. If $j>1$, then $\mathbb P$-almost all 
orderings are imperfect.
\end{theorem}

We motivate the proof of Theorem \ref{random_order_imperfect}  with the following remarks. 
For an $\omega\in\Omega$, 
in order that $s_\omega$ be extended continuously to $X$, 
it is necessary that for each $n$, there exists an $N$ such that 
knowledge of any path, $x$, in $X\setminus X_\text{max}$ up to level
$N$ determines $s_\omega(x)$ up to level $n$. Conversely to show 
$\omega$ is imperfect, one shows that there exists an $n$ such that
for each $N$, the first $N$ terms
of $x$ do not determine the first $n$ terms of $s(x)$. 
In fact, we show that for any $n$ and $N$, there exists a sequence of values,
$K$, such that if one considers the collection, $\mathcal M_{K}$ 
of finite paths from the $K$th 
level to the root that are non-maximal in the $K$th edge, but maximal in
all prior edges, the set of $\omega$ such that the first $N$ edges of $x$
determine the first $n$ edges of $s_\omega(x)$ is of measure 0.
The following lemma provides a key combinatorial estimate that we use
in the proof of  Theorem \ref{random_order_imperfect}. 
We make use of the obvious correspondence
between the collection of orderings on a set $S$ with $n$ elements, and
the collection of bijections from $\{1,2,\ldots,n\}$ to $S$.

\begin{lemma}\label{lem:Bing}
Let $S$ be a finite set of size $n$ and let $F$ and $G$ be 
maps from $S$ into a set $R$ with $G$ non-constant.
Let the set  $\Sigma$ of  total orderings on $S$ be
equipped with the uniform probability measure $\mathbb P$.

Then 
\begin{equation}\label{estimate}
\mathbb P(O)\le \tfrac 1{n-1}
\text{, where }O=\big\{\sigma \in \Sigma\colon F(\sigma(i))=G(\sigma(i+1))
\text{ for all $1\le i<n$}\big\}.
\end{equation}
\end{lemma}

\begin{proof}

Let $V$ be the union of the range of $F$ and the range of $G$. 
Form a directed multigraph $\mathcal G=(V,E)$ as follows. 
For $1\leq i\leq n$, define the ordered pair $e_i=(G(i),F(i))$. 
Let $E=\{ e_1,e_2,\dots ,e_n\}$.  Now let $\sigma\in O $. 
Then for $1\leq i<n$, the range of $e_{\sigma(i)}$ equals the source of 
$e_{\sigma(i+1)}$. Therefore, $e_{\sigma(1)}e_{\sigma(2)}\dots e_{\sigma(n)}$ 
is an Eulerian trail in $ \mathcal G$. 

It is straightforward to check that the map from $O$ to 
Eulerian trails is bijective, and thus we need to bound the number of 
Eulerian trails in $\mathcal G$.  
To do this, note that each Eulerian trail induces an ordering on the out-edges of each vertex. 
Let $V=\{ v_1,\dots , v_k\}$, and let $n_i$ be the number of out-edges 
of $v_i$. Since $G$ is non-constant, there are at least two directed 
edges with different sources, and thus $n_i\leq n-1$ for $1\leq i\leq k$. 
The number of orderings of out-edges is $n_1!n_2!\dots n_k!$. 

We distinguish two cases. If all vertices have out-degree equal to in-degree, then each Eulerian trail is in fact an Eulerian circuit. 
An Eulerian circuit corresponds to $n$ different Eulerian trails, 
distinguished by their starting edge. To count the number of circuits, 
we may fix a starting edge $e^*$, and then note that each circuit induces 
exactly one out-edge ordering if we start following the circuit at this edge. 
Note that in each such ordering, the edge $e^*$ must be the first in the 
ordering of the out-edges of its source. We may choose $e^*$ such that 
its source, say $v_1$, has maximum out-degree. Thus the number of 
compatible out-edge orderings is at most $(n_1-1)!n_2!\dots n_k!$
This expression is maximized, subject to the conditions  
$n_1+n_2+\dots +n_k\leq n$ and $n_i\leq n_1\leq n-1 $ for 
$1\leq i\leq k$,  when  $k=2$ and $n_1=n-1$, $n_2=1$. Therefore, 
there are at most $(n-2)!$ Eulerian circuits, so at most $n(n-2)!$ 
Eulerian trails and elements of $O$. 

If not all vertices have out-degree equal to in-degree, then either no 
Eulerian trail exists and the lemma trivially holds, or exactly one vertex, 
say $v_1$, has out-degree greater than in-degree, and this vertex must be 
the starting vertex of every trail. In this case, an ordering of out-edges 
from all vertices precisely determines the trail. The number of out-edge orderings 
(and good bijections) in this case is bounded above by $(n-1)!$.

Therefore, $O$ consists of at most $n(n-2)!$ orderings (out of $n!$), and the lemma follows.
\end{proof}

\begin{proof}[Proof of Theorem \ref{random_order_imperfect}]
Note that if $|V_n|=1$ for infinitely many $n$, then any order on $B$ has exactly 
one maximal and one minimal path. So we assume that $|V_n|\geq 2$ for all large $n$.

We first define some terminology. Recall that $s(e)$ and $r(e)$ denote the 
source and range  of the edge $e$ respectively. Given an order 
$\om \in \mathcal O_B$, we let $e_{\alpha,\om}(v)$ be the
$\alpha$th edge in $r^{-1}(v)$
If $v\in V_{N'}$ for some $N'>n$, 
we let $t_{n,\omega}(v)$
be the element of $V_n$ that the maximal incoming path to $v$ goes through.
We call $t_{n,\omega}(v)$ the $n$-\emph{tribe} of $v$. Similarly the $n$-\emph{clan}
of $v$, $c_{n,\omega}(v)$ is the element of $V_n$ through which the minimal 
incoming path to $v$ passes. 
If $n$ is such that for any $N>n$, the elements of $V_N$ belong to 
at least two $n$-clans (or $n$-tribes),
we shall say  that $\om$ has at least two infinite $n$-clans 
(or $n$-tribes.)
   
Let  $N>n$ and define $C_{n,N}$ to be the set of orders $\om$
such that if the non-maximal paths  
$x$ and $y$ agree  to level $N$, then their successors $s_\om(x)$ and  $s_\om(y)$  
agree to level $n$. Note that $\mathcal P_B \subset
\bigcap_{n=1}^\infty  \bigcup_{N=n}^{\infty}C_{n,N}$. In what follows we show that  
this last set has zero mass.
   
Fix $n$ and $N$  with $N>n$, and take any $N'>N$. 
Any order $\om\in C_{n,N}$ must satisfy the following constraints:
given any  two non-maximal edges whose sources in $V_{N'}$
belong to the same $N$-tribe,  their successors must belong 
to the same $n$-clan. In particular, 
if  $v$ and $v'$ are vertices in $V_{N'}$  such that the sources of
$e_{\alpha,\om}(v)$ and $e_{\beta,\om}(v)$
belong to the same $N$-tribe,  where $\alpha$ 
and $\beta$ are both non-maximal, then the  sources of 
$e_{\alpha+1,\om}(v)$ and $e_{\beta+1,\om}(v')$
must belong to the same $n$-clan. That is, there
is a map $f\colon V_{N}\to V_n$ such that for any $v\in V_{N'}$ 
and any non-maximal $\alpha$,
$f(t_{N,\om}(s(e_{\alpha,\om}(v))))=c_{n,\om}(s(e_{\alpha+1,\om}(v)))$.
We think of this $f$ as mapping $N$-tribes to $n$-clans. This is illustrated in
Figure \ref{fig:fcond}.

\begin{figure}
\includegraphics[width=2in]{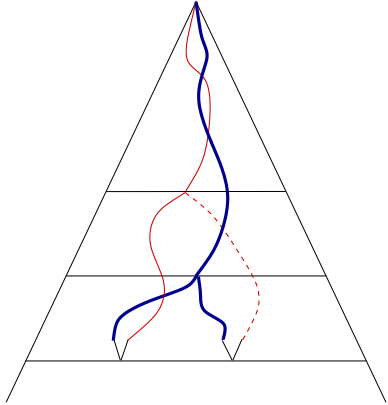}
\caption{The maximal upward paths from
  $s(e_{\alpha,\om}(v))$ and 
$s(e_{\beta,\om}(v'))$ (blue and bold) agree above level $N$, so the 
minimal upward path from
path with range $s(e_{\beta+1,\om}(v'))$ (red, dashed, not bold) 
must hit the same vertex in the $n$th level as the
minimal  upward path from
$s(e_{\alpha+1,\om}(v))$ (red, solid, not bold).}
\label{fig:fcond}
\end{figure}

Motivated by the preceding remark, if $N'>N>n$, we define two subsets of 
$\mathcal O_B$. 
We let $D_{n,N'}$ be the set of orders such that $V_{N'}$ 
contains members of at least two $n$-clans;
and $E_{n,N,N'}$ to be the subset of orders in $\mathcal D_{n,N'-1}$ which 
additionally satisfy the condition (*):
  
\begin{quote}
There is a function $f\colon V_N\to V_n$ such that for all $v \in V_{N'}$, 
if $\alpha$ is a non-maximal edge entering $v$ then
$f(t_{N,\omega}(s(e_{\alpha,\omega}(v))))=c_{n,\omega}(s(e_{\alpha+1,\omega}(v)))$.
\end{quote}
 
We observe that $D_{n,N'}$ and $E_{n,N,N'}$ 
are $\mathcal F_{N'}$-measurable.
We compute $\mathbb P(E_{n,N,N'}|\mathcal F_{N'-1})$. 
Since $D_{n,N'-1}$ is $\mathcal F_{N'-1}$ measurable, we 
have $\mathbb P(E_{n,N,N'}|\mathcal F_{N'-1})(\om)$ is 0
for $\om\not\in D_{n,N'-1}$. 
For a fixed map $f\colon V_N\to V_n$, and a fixed vertex $v\in V_{N'}$,
and $\om\in D_{n,N'-1}$, the conditional probability given $\mathcal F_{N'-1}$
that (*)
with the specific function $f$ is satisfied at $v$ is at most $1/(|V_{N'-1}|-1)$.
To see this, notice that for $\om\in D_{n,N'-1}$, the $n$-clan
is a non-constant function of $V_{N'-1}$, so that the hypothesis of 
Lemma \ref{lem:Bing} is satisfied,  with $F=f\circ t_{N,\om}\circ s$ and 
$G=c_{n,\om}\circ s$, both applied to the set of incoming edges to $v$.
Also, since $B$ is completely connected, there are at least $|V_{N'-1}|$ edges coming into $v$.

Since these are independent events conditioned on $\mathcal F_{N'-1}$, 
the conditional probability that (*) is satisfied for 
the fixed function $f$ over all $v\in V_{N'}$ is at most
$1/(|V_{N'-1}|-1)^{|V_N'|}$. There are $|V_n|^{|V_N|}$ possible functions
$f$ that might satisfy (*). 
Hence we obtain
$$
\mathbb P(E_{n,N,N'})\le \frac{|V_n|^{|V_N|}\mathbb P(D_{n,N'-1})}
{(|V_{N'-1}|-1)^{|V_{N'}|}},
$$
so that for fixed $n$ and $N$ with $n<N$, one has 
$\liminf_{N'\to\infty}\mathbb P(E_{n,N,N'})=0$.
By the hypothesis, for any $\epsilon>0$, there exists
$m(\epsilon)$ such that $\mathbb P(R_n)>1-\epsilon$ for all $n>m(\epsilon)$, where
$R_n=\{\om\in\mathcal O_B\colon \text{$\om$ has
at least 2 infinite $n$-clans}\}$.

Since  $C_{n,N}\cap R_n\subset E_{n,N,N'}$ for all $N'>N>n$, we conclude that 
$\mathbb P(C_{n,N}   \cap R_n )=0$ for $N>n$, so that 
$\mathbb P(C_{n,N}) \leq \epsilon$ for $ÊN>n>m(\epsilon)$.
Now since $\mathcal P_B \subset \bigcap_{n=1}^\infty  
\bigcup_{N=n}^{\infty}C_{n,N}$ and $C_{n,N} \subset C_{n,N+1}$ for each $N\geq n$, 
we conclude that $\mathbb P( \mathcal P_B)=0$.
\end{proof}

\section{Diagrams whose orders are almost always imperfect}\label{results}

\subsection{Bratteli diagrams and the Wright-Fisher model}\label{special_case} 
Let $\mathbf1_{a\times b}$ denote the $a\times b$ matrix
all of whose entries are 1. If $V_n$ is the $n$-th vertex set in B,
define $M_n= |V_n|$. In this sub-section, all Bratteli diagrams that we 
consider have incidence matrices $F_n=\mathbf 1_{M_{n+1}\times M_n}$ for each $n$.

We wish to give conditions on $(M_n)$ so that a $\PP$-random order has 
infinitely many maximal paths. We first comment on the relation between 
our question and the Wright-Fisher model in population genetics. 
Given a subset $A\subset V_k$, and an ordering $\omega \in \mathcal O_B$, 
we let $S^\omega_{k,n}(A)$ for $n>k$
be the collection of vertices $v$ in $V_n$ such that the unique 
upward maximal path in the $\omega$ ordering through
$v$ passes through $A$.   
Informally, when we consider the tree formed by all maximal edges
then $S^\om_{k,n}(A)$ is the set of vertices in $V_n$ that have 
``ancestors" in $A$.

Let $Y_n=|S^\om_{k,n}(A)|/M_n$. We observe that conditional on $Y_n$, 
each vertex in $V_{n+1}$ has probability $Y_n$ of belonging to
$S^\om_{k,n+1}(A)$ (since each vertex in $V_{n+1}$ chooses an
independent ordering of $V_n$ from the uniform distribution), so
that the distribution of $|S^{\om}_{k,n+1}(A)|$ conditional on $Y_n$
is binomial with parameters $M_{n+1}$ and $Y_n$.
In particular, $(Y_n)$ is a martingale with respect to the 
natural filtration $(\mathcal F_n)$, where $\mathcal F_n$ is the 
$\sigma$-algebra generated by  the $n$th level cylinder sets.
Since $(Y_n)$ is a bounded martingale, it follows from the 
martingale convergence theorem that $(Y_n)$  almost surely converges to 
some limit $Y_\infty$ where $0\leq Y_\infty \leq 1$. 

It turns out that the study of maximal paths 
is equivalent to the Wright-Fisher model in population
genetics. Here one studies populations where there are disjoint generations; each 
population member inherits an allele (gene type) from a uniformly randomly 
chosen member of the previous generation. 
To compare the randomly ordered Bratteli diagram and Wright-Fisher models,
the vertices in $V_n$ represent the $n$th generation and 
a vertex $v\in V_{n+1}$ ``inherits an allele"
from $w\in V_n$ if the edge from $w$ to $v$ is the maximal incoming edge to $v$. 
Since for each $v$, one of the $M_n!$ orderings of $V_n$ is chosen uniformly
at random, the probability that any element of $V_n$ is the source of the
maximal incoming edge to $v$ is $1/M_n$. Since the orderings are chosen
independently, the ancestor of $v\in V_{n+1}$ is independent of the ancestor
of any other $v'\in V_{n+1}$.

Analogous to $Y_n$, in the Wright-Fisher context, one studies 
the proportion of the population that 
have various alleles. If one declares the vertices in $A \subset V_k$ 
to have allele type \textbf{A} and the other vertices in that level 
to have allele type \textbf{a}, then there is a maximal path through $A$ if and only
if in the Wright-Fisher model, the allele \textbf{A} persists - that is there exist 
individuals in all levels beyond the $n$th with type \textbf{A} alleles.

In a realization of the Wright-Fisher model, an allele type is said to \emph{fixate}
if the proportion $Y_n$ of individuals with that allele type in the $n$th level 
converges to 0 or 1 as $n\to\infty$. An allele type is said to become 
\emph{extinct} if $Y_n=0$ for some finite level,
or to \emph{dominate} if $Y_n=1$ for some finite level.

\begin{theorem} \cite[Theorem 3.2]{donnelly} \label{thm:Donnelly} 
Consider a Wright-Fisher model with population structure $(M_n)_{n\ge 0}$. 
Then domination of one of the alleles occurs almost surely if and only if  
$\sum_{n\ge 0}1/M_n=\infty$.
\end{theorem}

Theorem \ref{thm:Donnelly} also holds if in the Wright-Fisher model, 
individuals can inherit one of $r$ alleles with $r\geq 2$.
We exploit this below by letting each member of a chosen generation 
have a distinct allele type. 

To indicate the flavour of the arguments, 
we give a proof of the simpler fact that if $\sum_{n\ge 0}1/M_n=\infty$
then each allele type fixates.  
To see this, let $Q_n=Y_n(1-Y_n)$. 
Now we have \[\E(Q_n|\mathcal F_{n-1})
=Y_{n-1}-Y_{n-1}^2-(\E(Y_n^2|Y_{n-1})-\E(Y_n|Y_{n-1})^2)
=Q_{n-1}-\Var(Y_n|Y_{n-1}).\]
Since $M_nY_n$ is binomial  with parameters $M_n$ and  $Y_{n-1}$, 
\[
\Var(Y_n|Y_{n-1})=
(1/M_n^2)(M_nY_{n-1}(1-Y_{n-1}))=Q_{n-1}/M_n.
\] 
This gives $\E(Q_n|\mathcal F_{n-1})=
(1-1/M_n)Q_{n-1}$. Now using the tower property of conditional expectations, we have
$\E Q_n=\E(Q_n|\mathcal F_0)=\prod_{j=1}^n(1-1/M_j)\E Q_0$, which converges to 0. 
As noted above, the sequence $(Y_n(\omega))$ is convergent for almost all 
$\omega$ to 
$Y_\infty(\omega)$
say. It follows that $Q_n(\omega)$ converges pointwise to $Y_\infty(1-Y_\infty)$. 
By the bounded convergence theorem, we deduce that $\E Y_\infty(1-Y_\infty)=0$,
so that $Y_\infty$ is equal to 0 or 1 almost everywhere.

We shall use Theorem \ref{thm:Donnelly}  to prove the first part of the 
following theorem.

\begin{theorem}\label{thm:dichot}
Consider a Bratteli diagram with $M_n\ge 1$ vertices in the $n$th 
level and whose incidence matrices are all of the form $\mathbf 1_{M_{n+1}\times M_n}$.
We have the following dichotomy:
 
If $\sum_n 1/M_n=\infty$, then there is $\mathbb P$-almost surely 
a unique maximal path.

If $\sum_n 1/M_n<\infty$, then there are $\mathbb P$-almost
surely uncountably many maximal paths. 
\end{theorem}

To prove the second part of this result we will need the following tool.
Recall the definition of  $S^\omega_{k,n}(A)$, from the
beginning of Section \ref{special_case}. 

\begin{proposition}\label{thm:mgslowdown}
Consider a Wright-Fisher model with population structure $(M_n)_{n\ge 0}$. 
Suppose that $\sum_{n\ge 0}1/M_n<\infty$. Then for each $\epsilon>0$ and 
$\eta>0$, there exists an $l>0$ such that for any $\mathcal F_l$-measurable
random subset, $A(\om)$, of $V_l$ (that is an $\mathcal F_l$-measurable
map $\Omega\to\mathcal P(V_l)$) and any $L>l$,
$$
\PP\left(\left|\frac{|A(\om)|}{|V_l|} - \frac{|S^\omega_{l,L}(A(\om))|}{|V_L|}\right|
\ge \eta\right)<\epsilon.
$$
\end{proposition}
 
\begin{proof}
Let $l$ be chosen so that $\sum_{n=l+1}^\infty 1/M_n<4\epsilon\eta^2$
and let $L>l$. 
For $n>l$, let $Y_n=\big|S^\omega_{l,n}\big(A(\omega)\big)\big|/|V_n|$.
Recall that $(Y_n)$ is a martingale with respect to the filtration
$(\mathcal F_n)$.
Set $Z_n=(Y_n-Y_l)^2$ and notice that $(Z_n)_{n\ge l}$ is a bounded 
sub-martingale by the conditional expectation version of Jensen's inequality. 

We have 
\begin{align*}
\E Z_L&=\E\big(\E((Y_L-Y_l)^2|\mathcal F_l)\big)\\
&=\E\big(\E(Y_L^2-Y_l^2|\mathcal F_l)\big)\\
&=\E(Y_L^2-Y_l^2)\\
&=\sum_{j=l}^{L-1} \E(Y_{j+1}^2-Y_{j}^2).
\end{align*}
A calculation shows that 
\begin{align*}
\E(Y_{j+1}^2-Y_{j}^2|\mathcal F_j)&=
\E(Y_{j+1}^2|\mathcal F_j)-\E(Y_{j+1}|\mathcal F_j)^2\\
&=\Var(Y_{j+1}|\mathcal F_j)\\
&=\frac{Y_j(1-Y_j)}{M_{j+1}}\\
\end{align*}
so that $\E(Y_{j+1}^2-Y_j^2)\le 1/(4M_{j+1})$ and 
we obtain $\E Z_L\le \sum_{j=l+1}^L 1/(4M_j)$. 
In particular we have $\E Z_{L}\le \epsilon\eta^2$. The claim follows
from Markov's inequality.

\end{proof}

\begin{proof}[Proof of Theorem \ref{thm:dichot}]
Suppose first that $\sum_n 1/M_n=\infty$. 
We show for all $k$, with probability 1, there exists $n>k$
such that all maximal paths from each level $n$ vertex to the root
vertex pass through a single vertex at level $k$. 

To do this, we consider the $M_k$ vertices at level $k$
to each have a distinct allele type. 
By Theorem \ref{thm:Donnelly}, there is for almost every 
$\omega$, a level $n$ such that by level $n$ one of the $M_k$ allele 
types has dominated all the others.   This is  a direct translation
of the statement that we need, which is that  every maximal 
finite path with range in $V_n$ passes through the same vertex in $V_k$.

Now we consider the case $\sum_n 1/M_n<\infty$. In this case, we identify a 
sequence $(n_k)$ of levels. We start with a single allele at $v_0$; and evolve
it to level $n_1$, where it is split into two almost equal  sub-alleles. This 
\emph{evolve-and-split} operation is repeated inductively, evolving the two
alleles at level $n_1$ to level $n_2$ and splitting each one giving 
four sub-alleles and so on, so that there are
$2^k$ alleles in the generations of the Bratteli diagram between the $n_k$th
and $n_{k+1}$st. We show that with very high probability, they all persist and
maintain a roughly even share of the population. This splitting allows us to find, with 
probability arbitrarily close to one,
a surjective map from the set of  maximal paths to all possible 
sequences of 0s and 1s. 

Fix a small $\kappa>0$. Using Proposition \ref{thm:mgslowdown},
choose an increasing sequence of levels $(n_k)_{k\ge 1}$ with the property
that $M_{n_k}>4^k$ and that for any random $\mathcal F_{n_k}$-measurable subset, 
$A(\omega)$ of $V_{n_k}$,
one has with probability at least $1-\kappa 4^{-k}$,
\begin{equation}\label{eq:slowdensitychange}
\left|\frac{|S_{n_k,n_{k+1}}^\omega(A(\omega))|}{|V_{n_{k+1}}|}-
\frac{|A(\omega)|}{|V_{n_k}|}\right| < 4^{-(k+1)}.
\end{equation}

Let $n_0=0$ and let $A_\epsilon(\om)=V_0$ (here $\epsilon$ stands for the empty string). 
We inductively define a collection of $\mathcal F_{n_k}$-measurable subsets of $V_{n_k}$
indexed by strings of 0's and 1's of length $k$. Suppose that for each string $s$ of length
$k$, $A_s(\om)$ is a random $\mathcal F_{n_k}$-measurable subset of $V_{n_k}$. 
Then we let $A_{s0}(\om)$ be the first half of $S_{n_k,n_{k+1}}^\om(A_s(\om))$
and $A_{s1}(\om)$ be the second half (by the first half of a subset $A$ of $V_n$,
we mean the subset consisting of the 
first $\lceil \frac {|A|}2\rceil$ elements of $A$ with respect
to the fixed indexing of $V_n$ and the second half is the subset 
consisting of the last $\lfloor \frac {|A|}2\rfloor$ elements of $A$).
By the union bound, we see that with probability at least $1-(\sum_{k=1}^\infty 2^k
\kappa 4^{-k})=1-\kappa$, the sets satisfy for each $s\in \{0,1\}^k$,
\begin{equation}\label{eq:nonwand2}
\left|\frac{|S^\om_{n_k,n_{k+1}}(A_{s}(\om))|}{|V_{n_{k+1}}|}-
\frac{|A_s(\om)|}{|V_{n_k}|}\right|<4^{-(k+1)}.
\end{equation} 
In particular, this suffices to ensure that the sets $A_s(\om)$ are non-empty for each
finite string of 0's and 1's. Now we define a map from the collection of maximal paths
to $\{0,1\}^{\mathbb N}$: for each $k$, the $A_s(\om)$ for $s\in\{0,1\}^k$ partition 
$V_{n_k}$. Given $x\in X_\text{max}(\om)$, there is a unique sequence $\iota(x)=
i_1i_2\ldots \in \{0,1\}^{\mathbb N}$ such that the $k_n$th edge, $r(x_{k_n})\in 
A_{i_1\ldots i_n}(\om)$. The map $\iota\colon X_\text{max}(\om)\to\{0,1\}^\N$
is then continuous. For any $\omega$ satisfying \eqref{eq:nonwand2},
the map $\iota\colon X_\text{max}(\om)\to \{0,1\}^\N$ is surjective. 
Hence for each $\kappa>0$, we have exhibited a measurable subset of $\Omega$ 
with measure $1-\kappa$ for which there are uncountably many maximal paths. 
By completeness of the measure, it follows that
almost every $\omega$ has uncountably many maximal paths.

\end{proof}

\subsection{Other Bratteli diagrams whose orders support many maximal paths}

Next we partially extend the results in Section \ref{special_case} 
to  a larger family of Bratteli diagrams.

\begin{definition} \label{equal_path_number}
Let $B$ be a Bratteli diagram. 
\begin{itemize}
\item We say that $B$ is {\em superquadratic} if there exists $\delta >0$ 
so that  $M_n \geq n^{2+\delta}$ for all large $n$.
\item
Let $B$ be superquadratic with constant $\delta$. We say that $B$ is 
{\em exponentially bounded} if
$\sum_{n=1}^{\infty} |V_{n+1}|\exp(- |V_n|/n^{2 + 2\delta/3}) $ converges. 

\end{itemize}
\end{definition}

We remark that the condition that $B$ is exponentially bounded is very mild.

In Theorem \ref{general_case} below we show that Bratteli diagrams satisfying 
these conditions have infinitely many maximal paths. 
Given $v\in V_{n+1}$, define
\[
V_{n}^{v,i}:=\{ w \in V_n: f_{v,w}^{(n)}= i\}\, , 
\]
so that if the incidence matrix entries for $B$ are all positive and bounded
above by $r$, then $V_n = \bigcup_{i=1}^r V_n^{v,i} \mbox{ for each } v \in V_{n+1}$.

\begin{definition}    Let $B$ be a Bratteli diagram with positive incidence matrices. 
We say that $B$ is {\em impartial} if there exists an integer $r$ so that all of $B$'s  
incidence matrix entries are bounded above by $r$, and if there exists some 
$\alpha \in (0,1)$ such that for any $n$, any $i\in \{1, \ldots , r\} $  and any 
$v\in V_{n+1}$,  $|V_n^{v,i}|\geq \alpha |V_n|$.
\end{definition}

In other words, $B$ is impartial if for any row of any incidence matrix, 
no entry occurs disproportionately rarely or often with respect to the others. 
For example, fixing $r$,  if we let $|V_n|=r(n+1)$, and  let each row 
of $F_n$ consist of  any vector with entries equidistributed from 
$\{1,\ldots, r\}$, the resulting Bratteli diagram  is impartial.
Note that our diagrams in Theorem  \ref{thm:dichot}
are impartial. 
However the vertex sets can grow as fast 
as we want, so the diagrams are not necessarily exponentially bounded. 
We remark also that if a Bratteli diagram is impartial, then it is completely
connected, which means that we can apply Theorem 
\ref{random_order_imperfect} if $j>1$.

\begin{definition} Suppose that $B$ is a Bratteli diagram each of whose incidence
matrices has entries with a maximum value of  $r$.  We say that 
$A\subset V_n$ is {\em $(\beta,\epsilon)$-equitable} for $B$ if for each 
$v\in V_{n+1}$ and for each $i=1, \ldots , r$, 
\[  
\left| \frac{|V_n^{v,i}\cap A|}{|V_n^{v,i}|} - \beta   \right|\leq \epsilon.   
\]

In the case $\beta=\frac 12$, we shall speak simply of $\epsilon$-equitability. 
\end{definition}

Given $v\in V\backslash V_0$ and an order $\om\in \mathcal O_B$, recall that we use
$\wt e_v = \wt e_v(\om)$ to denote the maximal edge with range $v$.

\begin{lemma}\label{mean_roughly_beta}
Suppose that $B$ is impartial. Let $A\subset V_n$ be $(\beta,\epsilon)$-equitable, 
and $v\in V_{n+1}$. Let the random variable $X_v$ be defined as
\begin{equation*} X_v(\om)
   = \left\{
\begin{array}{rl}
1
& \mbox{if $ s(\wt e_v)\in A$,
   } \\
0 & \mbox{ otherwise.}
  \end{array}
  \right. \end{equation*} 
Then $\beta- \epsilon \leq \mathbb E (X_v) \leq \beta+ \epsilon $.
\end{lemma}

\begin{proof}
We have
\begin{align*}
\mathbb E (X_v) &= \frac {\sum_{j=1}^{r} j|A\cap V_n^{v,j}|}{\sum_{j=1}^r j|V_n^{v,j}|}\\
&\leq   
\frac{\sum_{j=1}^{r} j| V_n^{v,j}|(\beta+\epsilon)}
{\sum_{j=1}^r j|V_n^{v,j}|}  =   
\beta+\epsilon,
\end{align*}
the last inequality following since $A$ is $\epsilon$-equitable. Similarly, 
$\mathbb E (X_v) \geq \beta-\epsilon$.
\end{proof}

\begin{lemma}\label{lem:Hoeffding-app}
Let $B$ be an impartial Bratteli diagram with impartiality constant $\alpha$
and the property that each entry of each incidence matrix is between 1 and $r$. 
Let $\beta$, $\delta$ and $\epsilon$ be positive,  
let $(p_v)_{v\in V_N}$ satisfy $|p_v-\beta|<\delta$ for each $v\in V_N$
and let $A\subset V_N$ be a randomly chosen subset, where each $v$
is included with probability $p_v$ independently of the inclusion
of all other vertices. Then the probability that $A$ fails to be
$(\beta,\delta+\epsilon)$-equitable is at most $2r|V_{N+1}|e^{-\alpha|V_N|\epsilon^2}$.
\end{lemma}

\begin{proof}
Let $(Z_v)_{v\in V_N}$ be $\mathbf 1_{v\in A}$, so that these are
independent Bernoulli random variables, where $Z_v$ takes the value 1 with probability $p_v$

For $u\in V_{N+1}$ and $1\leq i \leq r$, define
\begin{equation}\label{definition_Yw,i}
Y_{u,i} := \frac{1}{|V_{N}^{u,i}|}\sum_{v\in V_{N}^{u,i} } Z_{v}  
=\frac{ |  \{ v\in V_{N}^{u,i}: v\in A    \}|        }{|V_{N}^{u,i}|}
=\frac{  |A\cap V_{N}^{u,i}|      }{|V_{N}^{u,i}|} \, .  
\end{equation}
 
Using Hoeffding's inequality \cite{hoeffding}, since $\beta-\delta\leq \mathbb E (Y_{u,i} )
\leq \beta+\delta$  we have that
\begin{align*}
\mathbb P (\{|Y_{u,i}- \beta|\geq (\delta + \epsilon)  \})& \leq 
\mathbb P (\{   |  Y_{u,i}-   \mathbb E (Y_{u,i}) |\geq \epsilon   \}) \\ 
& \leq 
 2 e^{-2|V_{N}^{u,i}|\epsilon^2} \leq 2 e^{-2\alpha|V_{N}|\epsilon^2}.
\end{align*}
This implies that
\begin{equation}\label{distributed_a}
\mathbb P \left(
\bigcup_{i=1}^{r}\bigcup_{u\in V_{N+1}}\{|Y_{u,i}- \beta|\geq \delta+\epsilon    \}
\right)   \leq  2r|V_{N+1}| e^{-2|V_{N}|\alpha \epsilon^2}.
\end{equation}

\end{proof}

\begin{lemma}\label{distributed}
Suppose that $B$ is impartial, superquadratic and exponentially bounded.
Then for any $\epsilon$ small there exist $n$ and $A\subset V_n$ 
such that $A$ is $(\frac12,\epsilon)$-equitable.
\end{lemma}

\begin{proof}
Let $r$ and $\alpha$ be as in the statement of Lemma \ref{lem:Hoeffding-app} and 
apply that lemma with $p_v=\frac 12$ for each $v\in V_n$. By the superquadratic
and exponentially bounded properties, one has
$2r|V_{n+1}|e^{-2\alpha|V_n|\epsilon^2}<1$ for large $n$. 
Since the probability that a randomly chosen 
set is $(\frac12,\epsilon)$-equitable is positive, the existence of such a set 
is guaranteed.
\end{proof}

\begin{theorem}\label{general_case}
Suppose that  $B$ is a  Bratteli diagram that is  impartial, superquadratic and exponentially 
bounded. Then $\mathbb P$-almost all orders on $B$ have infinitely many maximal paths.
\end{theorem}

We note that in the special case where $B$ is defined as in Section 
\ref{special_case}, the following proof can be simplified and does not require the 
condition that $B$ is exponentially bounded. Instead of beginning our 
procedure with an equitable set, which is what we do below, we can start 
with any set $A_N\subset V_N$  whose size relative to $V_N$ is around 1/2.

\begin{proof}
Since $B$ is superquadratic, we find  a sequence $(\epsilon_j)$ 
such that 
\begin{eqnarray}\label{assumptions_on_epsilon}
&&\sum_{j=1}^{\infty}\epsilon_j <\infty  \text{ and } \label{epsilon1} \\
\label{assumptions_on_epsilon_2}
&&M_j\epsilon_{j}^{2} \geq j^{\gamma} \text{ for some }\gamma >0\text{ and large enough }j.
\label{epsilon2}  
\end{eqnarray}

Fix $N$ so that (\ref{epsilon2}) holds for all $j\geq N$, and let $N$ be large 
enough so that  $ \sum_{j=N}^{\infty}\epsilon_j<\frac{1}{2} $. Moreover, we can 
also choose our sequence $(\epsilon_j)$ and our $N$ large enough so that there exists
a set $A_N\subset V_N$ which is $\epsilon_N$-equitable: by Lemma \ref{distributed}, 
this can be done. For all $k\geq 0$, define also 
\[
\delta_{N+k}=\sum_{i=0}^{k} \epsilon_{N+i}.
\]
Finally, let $r$ be so that all entries of all $F_n$ are bounded above by $r$.

Define recursively, for all integers $k>0$ and all $v\in V_{N+k}$, the Bernoulli random 
variables $\{X_v: \mathcal O_B \rightarrow \{0,1\}: v\in V_{N+k} \}$,   
and the random sets $\{ A_{N+k}: \mathcal O_B \rightarrow 2^{V_{N+k}}: k\geq 1\}$, where  
$X_v(\om)=1$  if 
$s(\wt e_v) \in A_{N+k-1}$,  and 0 otherwise, and  $A_{N+k}=\{ v\in V_{N+k}\,:\, X_v=1\}$.

We shall show that for a large set of $\om$, each set $A_{N+k}$ is $\delta_{N+k}$-equitable. 
This implies that the size of $A_{N+k}$ is not far from $\frac12 |V_{N+k}|$.
For, if  $k\geq 1$, define the event 
\[
D_{N+k} := \{ \om : A_{N+k}\text{ is }\delta_{N+k}-\text{equitable} \}.
\]

We claim that 
$$
\mathbb P(D_{N+k+1}|D_{N+k})\ge 1-2r|V_{N+k+2}|e^{-2\alpha|V_{N+k+1}|\epsilon_{N+k+1}^2}.
$$
To see this, notice that if $\omega\in D_{N+k}$, 
then by Lemma \ref{mean_roughly_beta}, given $\mathcal F_{N+k}$, each 
vertex in $V_{N+k+1}$ is independently present in $A_{N+k+1}$ with probability in the 
range $[\beta-\delta_{N+k},\beta+\delta_{N+k}]$. Hence by Lemma
\ref{lem:Hoeffding-app}, $A_{N+k+1}$ is $\delta_{N+k+1}$-equitable with probability
at least $1-2r|V_{N+k+2}|e^{-\alpha|V_{N+k+1}|\epsilon_{N+k+2}^2}$.

Next we show that our work implies that a random order has at least two 
maximal paths. Let   $\gamma =\frac12- \sum_{j=N}^{\infty}\epsilon_j$. Notice 
that if $A_n\ne V_n$ for all $n>N$, then there are at least two maximal paths.
By our choice of $N$ and $\gamma >0$ we have that 
\begin{align*}
 \mathbb P(\{\om: |X_\text{max}(\om)|\geq 2\}) 
 &\geq   \mathbb P\left(\bigcap_{k=1}^{\infty}\left\{\om:  \gamma \leq 
 \frac{|A_{N+k}|}{|V_{N+k}|}\leq 1 - \gamma \,\, \right\}\right)\\
 &\geq   \mathbb P\left(\bigcap_{k=1}^{\infty}D_{N+k}\right)\\
 & =  \lim_{n\rightarrow \infty}   \mathbb P(D_{N+1})
   \prod_{k=1}^n\mathbb P(D_{N+k+1}\vert D_{N+k})\\
 & \geq    \lim_{n\rightarrow \infty} \mathbb P(D_{N+1}) 
  \prod_{k=1}^n   (1- 2r |V_{N+k+2}|e^{-2|V_{N+k+1}|\alpha\epsilon_{N+k+1}^2  }),
\end{align*}
and the condition that $B$ is superquadratic and exponentially bounded
ensures that this last term converges to a non-zero value.

We can repeat this argument to show that for any natural $k$, 
a random order has at least $k$ maximal paths.  We remark also that the 
techniques of Section \ref{special_case} could be generalized to show 
that a random order would have uncountably many maximal paths.
 
We now apply Theorem \ref{random_order_imperfect}. 
\end{proof}
\proof[Acknowledgements]
We thank Richard Nowakowski and Bing Zhou for helpful discussions around Lemma \ref{lem:Bing}.

{\footnotesize
\bibliographystyle{alpha}
\bibliography{bibliography}

\def\ocirc#1{\ifmmode\setbox0=\hbox{$#1$}\dimen0=\ht0 \advance\dimen0
  by1pt\rlap{\hbox to\wd0{\hss\raise\dimen0
  \hbox{\hskip.2em$\scriptscriptstyle\circ$}\hss}}#1\else {\accent"17 #1}\fi}
\begin{thebibliography}{BKY14}

\bibitem[BKY14]{bky}
Sergey Bezuglyi, Jan Kwiatkowski, and Reem Yassawi.
\newblock Perfect orderings on finite rank {B}ratteli diagrams.
\newblock {\em Canad. J. Math.}, 66(1):57--101, 2014.

\bibitem[Don86]{donnelly}
Peter Donnelly.
\newblock A genealogical approach to variable-population-size models in
  population genetics.
\newblock {\em J. Appl. Probab.}, 23(2):283--296, 1986.

\bibitem[Eff81]{effros}
Edward~G. Effros.
\newblock {\em Dimensions and {$C^{\ast} $}-algebras}, volume~46 of {\em CBMS
  Regional Conference Series in Mathematics}.
\newblock Conference Board of the Mathematical Sciences, Washington, D.C.,
  1981.

\bibitem[GPS95]{g_p_s}
Thierry Giordano, Ian~F. Putnam, and Christian~F. Skau.
\newblock Topological orbit equivalence and {$C^*$}-crossed products.
\newblock {\em J. Reine Angew. Math.}, 469:51--111, 1995.

\bibitem[Hoe63]{hoeffding}
Wassily Hoeffding.
\newblock Probability inequalities for sums of bounded random variables.
\newblock {\em J. Amer. Statist. Assoc.}, 58:13--30, 1963.

\bibitem[HPS92]{hps}
Richard~H. Herman, Ian~F. Putnam, and Christian~F. Skau.
\newblock Ordered {B}ratteli diagrams, dimension groups and topological
  dynamics.
\newblock {\em Internat. J. Math.}, 3(6):827--864, 1992.

\bibitem[Sen74]{seneta}
E~Seneta.
\newblock A note on the balance between random sampling and population size (on
  the 30th anniversary of {G}. {M}al\'ecot's paper).
\newblock {\em Genetics}, 77(3):607--610, 1974.

\bibitem[Ver81]{vershik_1}
A.~M. Vershik.
\newblock Uniform algebraic approximation of shift and multiplication
  operators.
\newblock {\em Dokl. Akad. Nauk SSSR}, 259(3):526--529, 1981.

\bibitem[Ver85]{vershik}
A.~M. Vershik.
\newblock A theorem on the {M}arkov periodical approximation in ergodic theory.
\newblock {\em J. Sov. Math.}, 28:667--674, 1985.

\end{thebibliography}
}

\end{document}